\documentclass[10pt,reqno]{amsart}

\usepackage[a4paper,left=35mm,right=35mm,top=30mm,bottom=30mm,marginpar=25mm]{geometry}

\usepackage{amsmath}
\usepackage{amssymb}
\usepackage{amsthm}
\usepackage[latin1]{inputenc}
\usepackage{eurosym}
\usepackage[dvips]{graphics}
\usepackage{graphicx}
\usepackage{epsfig}
\usepackage{hyperref}
\usepackage{dsfont}
\usepackage[nocompress, space]{cite}
\usepackage{color}
\usepackage{mathtools}
\usepackage[displaymath,mathlines]{lineno}

\allowdisplaybreaks
\usepackage{enumitem}

\usepackage{ifthen}

\makeindex

\newcommand{\Rr}{{\mathbb{R}}}

\newcommand{\Nn}{{\mathbb{N}}}

\newcommand{\Tt}{{\mathbb{T}}}

\newcommand{\epsi}{\varepsilon}

\def\leq{\leqslant}
\def\geq{\geqslant}

\numberwithin{equation}{section}

\newtheoremstyle{thmlemcorr}{10pt}{10pt}{\itshape}{}{\bfseries}{.}{10pt}{{\thmname{#1}\thmnumber{
#2}\thmnote{ (#3)}}}
\newtheoremstyle{thmlemcorr*}{10pt}{10pt}{\itshape}{}{\bfseries}{.}\newline{{\thmname{#1}\thmnumber{
#2}\thmnote{ (#3)}}}
\newtheoremstyle{defi}{10pt}{10pt}{\itshape}{}{\bfseries}{.}{10pt}{{\thmname{#1}\thmnumber{
#2}\thmnote{ (#3)}}}
\newtheoremstyle{remexample}{10pt}{10pt}{}{}{\bfseries}{.}{10pt}{{\thmname{#1}\thmnumber{
#2}\thmnote{ (#3)}}}
\newtheoremstyle{ass}{10pt}{10pt}{}{}{\bfseries}{.}{10pt}{{\thmname{#1}\thmnumber{
A#2}\thmnote{ (#3)}}}

\theoremstyle{thmlemcorr}
\newtheorem{theorem}{Theorem}
\numberwithin{theorem}{section}
\newtheorem{lemma}[theorem]{Lemma}
\newtheorem{corollary}[theorem]{Corollary}
\newtheorem{proposition}[theorem]{Proposition}

\theoremstyle{thmlemcorr*}
\newtheorem{theorem*}{Theorem}
\newtheorem{lemma*}[theorem]{Lemma}
\newtheorem{corollary*}[theorem]{Corollary}
\newtheorem{proposition*}[theorem]{Proposition}
\newtheorem{problem*}[theorem]{Problem}
\newtheorem{conjecture*}[theorem]{Conjecture}

\theoremstyle{defi}
\newtheorem{definition}[theorem]{Definition}

\newtheorem{claim}{Claim}

\theoremstyle{remexample}
\newtheorem{remark}[theorem]{Remark}

\newtheorem{teo}[theorem]{Theorem}
\newtheorem{lem}[theorem]{Lemma}

\theoremstyle{ass}

\begin{document}

\title[uniqueness structure of weakly coupled HJ systems]{ Uniqueness structure of weakly coupled systems of ergodic problems of Hamilton-Jacobi equations}

\author{Kengo Terai}
\address[K. Terai]{
        Department of Pure and Applied Mathematics, Graduate School of Fundamental Science and Engineering, Waseda University, Okubo 3-4-1, Shinjuku, Tokyo 169-8555, Japan.}
\email{ken5terai@akane.waseda.jp}

\keywords{Hamilton-Jacobi equations; Weakly coupled systems; Viscosity solutions; Nonlinear adjoint methods. }
\subjclass[2010]{
        35F21, 
        35A50, 
        37J50} 

\date{\today}

\begin{abstract}
Here, we address a uniqueness structure of viscosity solutions for ergodic problems of weakly coupled Hamilton-Jacobi systems. In particular, we study comparison principle with respect to generalized Mather measures as a generalization of \cite{MT}, which addressed the case of a single equation. To get the main result, it is important to construct Mather measures effectively. We overcome this difficulty by nonlinear adjoint methods. 

\end{abstract}

\maketitle

\section{Introduction}

In this paper, we consider the following weakly coupled Hamilton-Jacobi system: 
\begin{equation}\label{EP}
H(x,Dv(x,i),i)+\sum_{j=1}^m c_{ij}(v(x,i)-v(x,j))=\lambda \quad\mathrm{in}\  \Tt^d\times I,
\end{equation}
where $\Tt^d$ is the $d$-dimensional flat torus and we set $I:=\{1,2,...,m\}$, for fixed $m\in \Nn$. 
Here, $v: \Tt^d\times I \to \Rr$ and $\lambda \in \Rr$ is a pair of unknowns for \eqref{EP}.
For $i,j\in I$, $c_{ij}$ are given nonnegative constants and the Hamiltonian $H:\Tt^d \times \Rr^d \times I \to \Rr$ is a given function in $C^2(\Tt^d \times \Rr^d)$ for all $i\in I$, satisfying the following properties:

\begin{enumerate}[label=(A\arabic*)]
\item For every $x \in \Tt^d$ and $i \in I$, $p \mapsto H(x,p,i)$ is convex. \label{convex} 

\item
Uniformly for $x \in \Tt^d$ and $i\in I$,
\[ \lim_{|p| \to \infty}\frac{H(x,p,i)}{|p|}=\infty \quad \mathrm{and}\quad  \lim_{|p| \to \infty} \left( \frac{1}{2d}H(x,p,i)^2+D_xH(x,p,i)\cdot p \right)=\infty. \]
\label{coercive}

\item
There exists $C>0$ such that for all $(x,p,i)\in \Tt^d\times \Rr^d \times I$, $|D_xH(x,p,i)|\leq C(1+|p|^2)$.
\label{growth}

\item
 For all $i,j \in I$, $c_{ij}=c_{ji}$ .
\label{sym}
\end{enumerate}



 It is known that there exists unique $\lambda \in \Rr$ such that \eqref{EP} has viscosity solutions. Hence, by resetting the Hamiltonian, we can assume that the {\it ergodic constant} $\lambda=0$ without loss of generality; see \cite{cami2} and \cite{MT2}, for instance.

Weakly coupled Hamilton-Jacobi systems arise, for example, in the literature of optimal control problems with random switching costs, which are governed by specific Markov chains. These systems were discussed for a long time in the context of PDE theory; see \cite{davis}, \cite{FS}, and \cite{PW}, for instance. In particular, \cite{Enger} and \cite{Koike} established the framework of viscosity solutions for these systems.
To analyze the large time behavior of the solution for time-dependent problems, the ergodic problems like \eqref{EP} are derived; see \cite{cami2} and \cite{MT2}.

First, we recall the case of the single equation, that is $m=1$. It is known that \eqref{EP} has multiple solutions, not even up to constant (for example, \cite{MTbook} Chapter $6$). Therefore, it is important to investigate the structure of solutions for \eqref{EP}.  In \cite{fathi} and \cite{fathiantonio}, the above nonuniqueness phenomena were studied in the context of weak KAM theory. For a development of weak KAM theory, many researchers studied the structure of solutions and the large time behavior of the associated time-dependent problems; see \cite{fathi}, \cite{MTbook} and references therein.

In the last decade, weakly coupled Hamilton-Jacobi systems were studied from a view point of weak KAM theory. For example, \cite{cami2}, \cite{CGT}, and \cite{MT2}  investigated the large time behavior of the solution for time-dependent problems. In \cite{MT4}, the authors studied homogenization for weakly coupled systems and the rate of convergence to matched solutions. On the other hand, \cite{davi} generalized the notion of Aubry sets for the case of systems and proved comparison principle with respect to their boundary data on Aubry sets. In \cite{MSTY}, the authors characterized the subsolutions of the systems and showed explicit representation for subsolutions enjoying maximal property. We remark that \cite{CGT} and \cite{figalli} studied weakly coupled Hamilton-Jacobi system which is a different type from \eqref{EP}.

However, it is little known what kinds of conditions characterize uniqueness of solutions for \eqref{EP}. In the single case,  one approach, studied in \cite{fathi} and \cite{fathiantonio}, is to find a {\it uniqueness set}; that is, if two solutions coincide on this set, they are totally equal on the domain.
 Recently, in \cite{MT}, a new and simple way to find uniqueness sets was studied. In particular, the authors proved new comparison principle with respect to Mather measures in the single case.  In this paper, we consider the above for the case of a weakly coupled system as a generalization of \cite{MT}.

To present the main result, we recall the definition of a generalized Mather measure in \cite{gomes2}. 
Set $L:\Tt^d \times \Rr^d \times I \to \Rr$ by
\[L(x,q,i):=\sup_{p\in \Rr^d}\{p\cdot q-H(x,p,i)\}.\]
\begin{definition}
We define a generalized Mather measure associated with \eqref{EP} by a minimizer of the following minimizing problem:
\begin{equation}\label{mb}
\inf_{\mu \in \mathcal{F}} \int_{\Tt^d \times \Rr^d \times I} L(x,q,i) \mbox{ d}\mu(x,q,i),
\end{equation}
where $\mathcal{F}$ is the set of all Radon probability measures $P(\Tt^d \times \Rr^d \times I)$ satisfying,
\[\int_{\Tt^d \times \Rr^d \times I}q\cdot D\phi(x,i)+\Theta \phi(x,i)\mbox{ d}\mu(x,q,i)=0,\]
for all $\phi(\cdot,i) \in C^1(\Tt^d)$, where 
\[\Theta \phi (x,i):=\sum_{j=1}^m c_{ij}(\phi(x,i)-\phi(x,j)).\]

We denote the set of all generalized Mather measures by $\mathcal{\tilde M}$. 
\end{definition}

We remark that the infimum of \eqref{mb} is zero because we set $\lambda=0$. Indeed, we denote it later as Corollary \ref{mbp}.

The following is the main result of this paper:
\begin{teo}\label{result}
Let $v_1(x,i), v_2(x,i)$ be Lipschitz continuous viscosity solutions of \eqref{EP}. Assume that \ref{convex}-\ref{sym} hold. If
\[ \int_{\Tt^d \times \Rr^d \times I} v_1(x,i) \mbox{ }d\mu(x,q,i) \leq \int_{\Tt^d \times \Rr^d \times I} v_2(x,i) \mbox{ }d\mu(x,q,i), \]
for any $\mu \in \mathcal{\tilde M}$, then, $v_1(x,i) \leq v_2(x,i)$ for all $(x,i)\in \Tt^d \times I$.
\end{teo}

In this paper, we regard the index of $m$-components system as a variable $i\in I$. This is a successful setting to discuss the above comparison result. 
To prove this Theorem, we use nonlinear adjoint methods (established in \cite{evans}) which fit nicely with the system structure. Another key point is to consider the Cauchy problem, not \eqref{EP} itself, for the system with initial data being approximations of solutions to \eqref{EP}. It is important noting that solutions of Cauchy problems are still quite close to that of \eqref{EP} (see Proposition \ref{converge}). This way, we are able to use the large time averaging effect of the Cauchy problem to introduce the adjoint problems, and then construct Mather measures in our setting.

As well as \cite{MT}, in light of Theorem \ref{result}, we can see that
\[ \mathcal{M}:=\overline{\bigcup_{\mu \in \tilde{\mathcal{M}}}\mathrm{supp}(\mathrm{proj}_{\Tt^d \times I}\mu)} \subset \Tt^d \times I \]
is a uniqueness set, that is,
\[ v_1=v_2 \:\: \mathrm{in}\: \mathcal{M} \quad \Rightarrow \quad v_1=v_2 \:\: \mathrm{in}\: \Tt^d \times I. \] 

This paper is organized as follows. In Section $2$, we provide some basic Lemmas. In Section $3$, we prove the main Theorem using nonlinear adjoint methods. Finally, we show an example of a generalized Mather measure defined by above in Section $4$. 
\section{Preliminaries}
In this section, we study the related Cauchy problem and adjoint problem for \eqref{EP} and give a priori estimate of solutions as preliminaries.
\subsection{Some properties of $\Theta$.}
Because $c_{ij}$ is symmetric, it holds the following identities.
\begin{lem}
Assume that \ref{sym} holds. Let $f,g: \Tt^d\times I \to \Rr$. Then, 
\begin{equation}\label{Theta} 
\int_{I} f(x,i)\Theta g(x,i)\mbox{ }di=\int_ {I} g(x,i) \Theta f(x,i) \mbox{ }di.
\end{equation}
\end{lem}

\begin{proof} We have
\begin{align*}
\int_{I} f(x,i)\Theta g(x,i)\mbox{ }di&=\sum_{i=1}^m f(x,i)\sum_{j=1}^m c_{ij}(g(x,i)-g(x,j))\\ 
&=\sum_{i,j=1}^m c_{ij}(g(x,i)-g(x,j))f(x,i) \\ 
&=\sum_{i,j=1}^m c_{ij}g(x,i)f(x,i)-\sum_{i,j=1}^m c_{ij}g(x,j)f(x,i) \\ 
&=\sum_{i,j=1}^m c_{ij}g(x,i)f(x,i)-\sum_{i,j=1}^m c_{ij}g(x,i)f(x,j) \\ 
&=\sum_{i=1}^m g(x,i)\sum_{j=1}^m c_{ij}(f(x,i)-f(x,j))
=\int_ {I} g(x,i) \Theta f(x,i) \mbox{ }di,
\end{align*}
where we used \ref{sym} in the forth identity.
\end{proof}

\begin{lem}
Assume that \ref{sym} holds. Let $f: \Tt^d\times I \to \Rr$. Then, 
\begin{equation}\label{sum}
\int_I \Theta f(x,i)\mbox{ }di=0.
\end{equation}
\end{lem}

\begin{proof} Using \ref{sym} in the following third identity, we get \eqref{sum}:
\begin{align*}
\int_I \Theta f(x,i)\mbox{ }di&=\sum_{i,j=1}^m c_{ij}(f(x,i)-f(x,j))=\sum_{i,j=1}^m c_{ij}f(x,i)-\sum_{i,j=1}^m c_{ij}f(x,j)\\ \notag
&=\sum_{i,j=1}^m c_{ij}f(x,i)-\sum_{i,j=1}^m c_{ij}f(x,i)=0.
\end{align*}
\end{proof}

\subsection{Cauchy problem} Propositions in this subsection are obtained by standard arguments in the theory of viscosity solutions. However, we discuss them to make the paper self-contained. 

Let $v_l$ be viscosity solutions of \eqref{EP} for $l=1,2$. For $\delta>0$, set 
\begin{equation}\label{moli}
v^\delta_l(x,i):=\gamma^\delta*v_l(x,i)=\int_{\Rr^d} \gamma^\delta(y)v_l(x+y,i)\mbox{ }dy,
\end{equation}
 where $\gamma^\delta(y):=\delta^{-d}\gamma(\delta^{-1}y)$ for $y\in \Rr^d$ and $\gamma$ is a standard mollifier. Then, we get the following estimate.
\begin{lem}
Assume that \ref{coercive} and \ref{sym} hold. Let $v^\delta_l(x)$ defined as \eqref{moli}. Then, there exists $C>0$ independent of $\delta>0$ such that 
\begin{equation}\label{molibdd1}
\|Dv_l^\delta(\cdot,i)\|_{L^\infty(\Tt^d)}+\delta \|\Delta v_l^\delta(\cdot,i) \|_{L^\infty(\Tt^d)} \leq C.
\end{equation}
\end{lem}

\begin{proof}
Applying \ref{coercive} and \eqref{sum} to \eqref{EP}, we can estimate
\begin{equation}\label{Dv}
 \|Dv_l(\cdot,i)\|_{L^\infty(\Tt^d)}\leq C.
\end{equation}
Hence, for $k\in \{1,2,...,d\}$ and $x\in \Tt^d$, we get
\[|(v^\delta_l)_{x_k}(x,i)|=\big|\int_{\Rr^d}(\gamma^\delta)_{x_k}(y)v_l(x+y,i)\mbox{ }dy\big|\leq \|Dv_l(\cdot,i)\|_{L^\infty(\Tt^d)} \int_{\Rr^d} |\gamma^\delta(y)|\mbox{ }dy\leq C.\]
To get the latter estimate in \eqref{molibdd1}, we calculate, for $x \in \Tt^d$,
\begin{align*}
|\Delta v^\delta_l(x,i)|\leq \int_{\Rr^d}|D\gamma^\delta(y)\cdot Dv_l(x+y,i)|\mbox{ }dy\leq \frac{C}{\delta^{d+1}}\int_{\Rr^d} |D\gamma(\frac{y}{\delta})|\mbox{ }dy=\frac{C}{\delta}\int_{\Rr^d}|D\gamma(z)|\mbox{ }dz\leq \frac{C}{\delta}.
\end{align*}
\end{proof}

In this subsection, we consider the following Cauchy problems:
\begin{equation}\label{CPR}
\begin{cases}
&\epsi (u^\epsi_l)_t+H(x,Du^\epsi_l,i)+\Theta u^\epsi_l=\epsi^4 \Delta u^\epsi_l \quad\rm{in}\  \Tt^d\times (0,1)\times I, \\
&u^\epsi_l(x,0,i)=v^{\epsi^4}_l(x,i) \quad \rm{in} \ \Tt^d \times I,
\end{cases}
\end{equation}
 for $l=1,2$, $\epsi>0$ and let $v^{\epsi^4}_l$ defined as \eqref{moli} with $\delta=\epsi^4$.
\begin{equation}\label{CP}
\begin{cases}
& \epsi(w^\epsi_l)_t+H(x,Dw^\epsi_l,i)+\Theta w^\epsi_l=0 \quad \rm{in}\  \Tt^d\times (0,1)\times I, \\
&w^\epsi_l(x,0,i)=v_l(x,i) \quad \rm{in} \ \Tt^d\times I.
\end{cases}
\end{equation}
Let $u^\epsi_l(x,t,i)$ and $w^{\epsi}_l(x,t,i)$ be the unique classical solution and viscosity solution of \eqref{CPR} and \eqref{CP}, respectively. It is obvious that the unique viscosity solution to \eqref{CP} is $w^\epsi_l=v_l$.
First, we investigate the difference between $u^{\epsi}_l(x,t,i)$ and $w^{\epsi}_l(x,t,i)$.

\begin{proposition}\label{converge}
Assume that \ref{coercive}-\ref{sym} hold. Then,
\[ \lim_{\epsi \to 0}\|u^{\epsi}_l(\cdot,i)-w^\epsi_l(\cdot,i)\|_{L^\infty(\Tt^d \times[0,1])}=0.\]
\end{proposition}

\begin{proof}
To denote simply, we write $u^{\epsi}$ and $w^{\epsi}$ instead of $u^{\epsi}_l$ and $w^{\epsi}_l$. Define $\Phi: \Tt^d\times \Tt^d \times [0,1]\times I\to \Rr$ as 
\begin{equation*}
\Phi(x,y,t,i):=w^{\epsi}(x,t,i)-u^{\epsi}(y,t,i)-\frac{|x-y|^2}{2\eta}-Kt,
\end{equation*}
for $\eta>0$ and $K>0$ to be fixed later. Take $(x_0,y_0,t_0,i_0)\in \Tt^d\times \Tt^d\times [0,1] \times I$ such that $\Phi(x_0,y_0,t_0,i_0)=\max_{\Tt^d\times \Tt^d\times [0,1] \times I} \Phi$. We first prove

\begin{claim}
For sufficiently large $C'>0$, let 
  $K:=\frac{C'}{\epsi}\left(\eta+\frac{\epsi^4}{\eta}\right)$. Then, $t_0=0$.
 \end{claim}
 Suppose $0<t_0\leq1$. In light of Ishii's Lemma (see, \cite{users} Theorem 8.3), for any $\rho>0$, there exists $(a, p_0, X)\in \bar{J}^{2,+}w^{\epsi}(x_0,t_0,i_0)$ and $(b, p_0, Y)\in \bar{J}^{2,-}u^{\epsi}(y_0,t_0,i_0)$ such that 
 \[p_0:=\frac{x_0-y_0}{\eta},\]
 \[a-b=K,\]
 and
 \begin{equation}\label{matrix}
 \begin{pmatrix}
 X&0\\
 0 &-Y
 \end{pmatrix}
\leq \frac{1}{\eta}
\begin{pmatrix}
 I_n&-I_n\\
 -I_n &I_n
 \end{pmatrix}
+\frac{\rho}{\eta^2} 
\begin{pmatrix}
 I_n&-I_n\\
 -I_n &I_n
 \end{pmatrix}
 ^2.
 \end{equation}
 By the definition of viscosity solutions (see \cite{Koike}, Proposition2.3), for $0<t_0\leq 1$, we have
 \begin{equation*}
 \epsi b+H(y_0,p_0,i_0)+\Theta u^{\epsi}(y_0,t_0,i_0)\geq \epsi^4 \mathrm{tr}(Y),
 \end{equation*}
 and 
 \begin{equation*}
 \epsi a+H(x_0,p_0,i_0)+\Theta w^{\epsi}(x_0,t_0,i_0) \leq 0.
 \end{equation*}
  Hence
 \begin{equation}\label{ll}
 \epsi K+H(x_0,p_0,i_0)-H(y_0,p_0,i_0)+\Theta w^{\epsi}(x_0,t_0,i_0)-\Theta u^{\epsi}(y_0,t_0,i_0)\leq -\epsi^4\mathrm{tr}(Y).
 \end{equation}
 Note that, by \eqref{matrix},
 \begin{align*}
 -\epsi^4 \mathrm{tr}(Y)\leq \frac{C\epsi^4}{\eta}+C\rho,
 \end{align*}
 and
 \begin{align*}
 &\Theta w^{\epsi}(x_0,t_0,i_0)- \Theta u^{\epsi}(y_0,t_0,i_0)\\
 &=\sum_{j=1}^m c_{i_0j}\left\{ (w^{\epsi}(x_0,t_0,i_0)-w^{\epsi}(x_0,t_0,j))-(u^{\epsi}(y_0,t_0,i_0)-u^{\epsi}(y_0,t_0,j))\right\}\\
 &=\sum_{j=1}^m c_{i_0j}\left\{ (w^{\epsi}(x_0,t_0,i_0)-u^{\epsi}(y_0,t_0,i_0))-(w^{\epsi}(x_0,t_0,j)-u^{\epsi}(y_0,t_0,j))\right\}\\
 &=\sum_{j=1}^m c_{i_0j}\left\{ \Phi(x_0,y_0,t_0,i_0)-\Phi(x_0,y_0,t_0,j)\right\} \geq0.
 \end{align*}
 On the other hand, because $\Phi(y_0,y_0,t_0,i_0)\leq \Phi(x_0,y_0,t_0,i_0)$, we get 
 \[w^{\epsi}(y_0,t_0,i_0) -u^{\epsi}(y_0,t_0,i_0)-Kt_0 \leq w^{\epsi}(x_0,t_0,i_0)-u^{\epsi}(y_0,t_0,i_0)-\frac{|x_0-y_0|^2}{2\eta}-Kt_0,\]
 which implies $|p_0|\leq C$. Thus, $|x_0-y_0|\leq C\eta$. Therefore, in light of \ref{growth},
 \begin{equation*}
 |H(x_0,p_0,i_0)-H(y_0,p_0,i_0)|\leq C(1+|p_0|^2)|x_0-y_0|\leq C\eta.
 \end{equation*}
  Apply these estimates for \eqref{ll} to deduce
  \begin{equation*}
  \epsi K\leq C\eta+\frac{C\epsi^4}{\eta}+C\rho.
  \end{equation*}
Sending $\rho \to 0$ yields a contradiction, which finishes the proof of Claim $1$.\\ \vspace{3mm}

By the above claim, we get
\begin{equation*}
w^{\epsi}(x,t,i)-u^{\epsi}(x,t,i)-Kt=\Phi(x,x,t,i)\leq \Phi(x_0,y_0,0,i_0)\leq v(x_0,i_0)-v^{\epsi^4}(y_0,i_0).
\end{equation*}
Hence, we have
\begin{align*}
w^{\epsi}(x,t,i)-u^{\epsi}(x,t,i)
&\leq v(x_0,i_0)-v^{\epsi^4}(x_0,i_0)+v^{\epsi^4}(x_0,i_0)-v^{\epsi^4}(y_0,i_0)+K\\
&\leq o(1)+\|Dv^{\epsi^4}(\cdot,i_0)\|_{L^\infty(\Tt^d)}|x_0-y_0|+K.
\end{align*}
On the other hand, because $\Phi(y_0,y_0,0,i_0)\leq \Phi(x_0,y_0,0,i_0)$, we get $|x_0-y_0|\leq C\eta$.
Combine the above two inequalities, to imply 
\begin{equation*}
w^{\epsi}(x,t,i)-u^{\epsi}(x,t,i)\leq o(1)+ C\eta+K= o(1)+C\eta+\frac{C'}{\epsi}\left(\eta+\frac{\epsi^4}{\eta}\right).
\end{equation*}
Setting $\eta=\epsi^2$, it holds that 
\[w^{\epsi}(x,t,i)-u^{\epsi}(x,t,i)\leq o(1). \]
By symmetry, we obtain the opposite inequality.
\end{proof}
To prove the main result, Lipschitz bound for $u^{\epsi}$ is important. To get this, we prove the following Lemma.

\begin{lemma}\label{bdd1}
Assume that \ref{coercive}-\ref{sym} hold. There exists $C>0$ independent of $\epsi>0$ such that
\begin{equation}
\|\Theta u^{\epsi}_l(\cdot,i)\|_{L^\infty(\Tt^d\times[0,1])}+\|\epsi \frac{\partial u^{\epsi}_l}{\partial t}(\cdot,i)\|_{L^\infty(\Tt^d\times[0,1])}\leq C.
\end{equation}
\end{lemma}

\begin{proof}
First, from \eqref{EP} and \eqref{Dv}, we have $\|\Theta v_l(\cdot,i)\|_{L^\infty(\Tt^d)}\leq C$.
By Proposition \ref{converge}, we get
 $\|\Theta u^{\epsi}_l(\cdot,i)\|_{L^\infty(\Tt^d\times[0,1])}\leq C$. On the other hand, by  \eqref{molibdd1}, for a suitably large $C>0$, each $u^{\pm}_l(x,t,i):= v^{\epsi^4}_l(x,i)\pm \frac{C}{\epsi}t$ is a classical supersolution and subsolution for \eqref{CPR}, respectively. By comparison, we get $v^{\epsi^4}_l(x,i)- \frac{C}{\epsi}t\leq u^\epsi(x,t,i)\leq v^{\epsi^4}_l(x,i)+ \frac{C}{\epsi}t$ for any $(x,t,i)\in \Tt^d\times [0,1] \times I$. Use comparison again to yield
\begin{equation*}
u^\epsi (x,t+s,i)-u^\epsi(x,t,i)\leq \max_{x\in \Tt^d}|u^\epsi(x,s,i)-v^{\epsi^4}_l(x,i)|.
\end{equation*}
Hence, we get $\|\epsi \frac{\partial u^\epsi}{\partial t}(\cdot,i)\|_{L^\infty(\Tt^d\times[0,1])}\leq C$.
\end{proof}

By Lemma \ref{bdd1}, we get Lipschitz bound for $u^{\epsi}$ using Bernstein's method.

\begin{proposition} \label{bernstein}
Assume that \ref{coercive}-\ref{sym} hold. There exists $C>0$ independent of $\epsi>0$ such that
\[\|Du^{\epsi}_l(\cdot,i)\|_{L^\infty(\Tt^d\times[0,1])}\leq C.\]
\end{proposition}

\begin{proof}
Take $k\in \{1,2,...,d\}$. In this proof, we denote $u^{\epsi}$ instead of $u_l^{\epsi}$. Differentiate \eqref{CPR} with respect to $x_k$ to get
\begin{equation*}\epsi u^\epsi_{tx_k}(\cdot,i)+H_{x_k}+D_pH\cdot Du^\epsi_{x_k}(\cdot,i)+\Theta u^\epsi_{x_k}(\cdot,i)=\epsi^4 \Delta u^\epsi_{x_k}(\cdot,i).
\end{equation*}
Multiplying by $u^\epsi_{x_k}(\cdot,i)$ and summing up with respect to $k$, we obtain
\begin{align*}
&\sum_{k=1}^d D_pH\cdot Du^\epsi_{x_k}(\cdot,i)u^\epsi_{x_k}(\cdot,i)+u^\epsi_{x_k}(\cdot,i)\Theta u^\epsi_{x_k}(\cdot,i)\\
& \:\:\:\:\: \epsi(\frac{1}{2}|Du^\epsi(\cdot,i)|^2)_t+D_xH\cdot Du^\epsi(\cdot,i)=\sum_{k=1}^d \epsi^4 \Delta u^\epsi_{x_k}(\cdot,i)u^\epsi_{x_k}(\cdot,i).
\end{align*}
Then, we can rewrite this as
\begin{align*}
&\epsi \psi_t +D_xH\cdot Du^\epsi(\cdot,i)+D_pH\cdot D\psi+\sum_{k=1}^d u^\epsi_{x_k}(\cdot,i)\Theta u^\epsi_{x_k}(\cdot,i)
=\epsi^4\{\Delta \psi -|D^2u^\epsi(\cdot,i)|^2\},
\end{align*}
where $\psi(x,t,i)=\frac{1}{2}  |Du^\epsi(x,t,i)|^2$.
Take $(x_0,t_0,i_0)\in \Tt^d\times[0,1]\times I$ as a maximum point of $\psi$. In the case $t_0=0$, it holds that
\[ \|Du^\epsi(\cdot,i)\|_{L^\infty(\Tt^d \times [0,1])}\leq \|Dv^{\epsi^4}(\cdot,i_0)\|_{L^\infty(\Tt^d)} \leq C. \]
Hence, it suffices to prove the case $t_0\in (0,1]$. Then,
\begin{equation}\label{asd}
\epsi^4|D^2u^\epsi(x_0,t_0,i_0)|^2+D_xH\cdot Du^\epsi(x_0,t_0,i_0) + \sum_{k=1}^d u^\epsi_{x_k}(x_0,t_0,i_0)\Theta u^\epsi_{x_k}(x_0,t_0,i_0)\leq0.
\end{equation}
Using Lemma \ref{bdd1}, we get 
\begin{align*}
&\epsi^4|D^2u^\epsi(x_0,t_0,i_0)|^2\geq \epsi^8|D^2u^\epsi(x_0,t_0,i_0)|^2 \geq \frac{\epsi^8}{d}|\Delta u^\epsi(x_0,t_0,i_0)|^2\\
&=\frac{1}{d}\left\{\epsi u^\epsi_t(x_0,t_0,i_0)+H(x_0,Du^\epsi(x_0,t_0,i_0)i_0)+\Theta u^\epsi(x_0,t_0,i_0)\right\}^2
\geq \frac{1}{2d}H(x_0,Du^\epsi(x_0,t_0,i_0),i_0)^2-C.
\end{align*}
Applying the above inequality to \eqref{asd}, we get 
\begin{equation}\label{bsd}
\frac{1}{2d}H(x_0,Du^\epsi(x_0,t_0,i_0),i_0)^2+D_xH\cdot Du^\epsi(x_0,t_0,i_0)+ \sum_{k=1}^d u^\epsi_{x_k}(x_0,t_0,i_0)\Theta u^\epsi_{x_k}(x_0,t_0,i_0) \leq C.
\end{equation}
Also, we can see 
\begin{align*}
& \sum_{k=1}^d u^\epsi_{x_k}(x_0,t_0,i_0)\Theta u^\epsi_{x_k}(x_0,t_0,i_0)
 =\sum_{k=1}^d \sum_{j=1}^m c_{i_0j}u^\epsi_{x_k}(x_0,t_0,i_0) \big\{u^\epsi_{x_k}(x_0,t_0,i_0)-u^\epsi_{x_k}(x_0,t_0,j)\big\} \\
& = \sum_{k=1}^d \sum_{j=1}^m c_{i_0j} \left\{ \frac{1}{2}\big(u^\epsi_{x_k}(x_0,t_0,i_0)\big)^2-u^\epsi_{x_k}(x_0,t_0,i_0)u^\epsi_{x_k}(x_0,t_0,j)+\frac{1}{2}\big(u^\epsi_{x_k}(x_0,t_0,i_0)\big)^2 \right\} \\
& \geq \sum_{k=1}^d \sum_{j=1}^m c_{i_0j} \left\{ \frac{1}{2}\big(u^\epsi_{x_k}(x_0,t_0,i_0)\big)^2-u^\epsi_{x_k}(x_0,t_0,i_0)u^\epsi_{x_k}(x_0,t_0,j)+\frac{1}{2}\big(u^\epsi_{x_k}(x_0,t_0,j)\big)^2 \right\} \\
& =\sum_{k=1}^d \sum_{j=1}^m \frac{1}{2}c_{i_0j} \big\{u^\epsi_{x_k}(x_0,t_0,i_0)-u^\epsi_{x_k}(x_0,t_0,j)\big\}^2 \geq0.
\end{align*} 
Hence, we obtain
\begin{equation*}
  \frac{1}{2d}H(x_0,Du^\epsi(x_0,t_0,i_0),i_0)^2+D_xH\cdot Du^\epsi(x_0,t_0,i_0)\leq C.
\end{equation*}
In light of \ref{coercive}, we get the conclusion.
\end{proof}

\subsection{Adjoint problem}
Fix $x_0\in \Tt^d$ and $k \in I$. Then, consider 
\begin{equation}\label{AJ}
\begin{cases}
& -\epsi \sigma^\epsi_t-\mathrm{div}(\sigma^\epsi D_pH(x,Du^{\epsi}_2,i))+\Theta \sigma^\epsi=\epsi^4 \Delta \sigma^\epsi \quad\rm{in}\  \Tt^d\times (0,1)\times I, \\
&\sigma^\epsi(x,1,i)=\gamma_{ik}\delta_{x_0} \quad \rm{in} \ \Tt^d \times I,
\end{cases}
\end{equation}
 where $\gamma_{ik}$ denotes the Kronecker delta and $\delta_{x_0}$ is the Dirac delta mass at $x_0$. Let $\sigma^\epsi(\cdot, i)$ be the solution of this problem.
 
 In this subsection, we recall that $\sigma^\epsi$ is nonnegative and preserves its total mass.
 \begin{proposition}\label{l1}
 For all $t\in [0,1]$ and $\epsi>0$, $\sigma^\epsi \geq0$ and
 \begin{equation*}
 \int_{\Tt^d\times I} \sigma^\epsi(x,t,i) \mbox{ }dxdi =1.
 \end{equation*}
 \end{proposition}
 
 \begin{proof}
 First, we prove that $\sigma^\epsi$ is nonnegative. Let $z^\epsi$ solve
 \begin{equation}\label{RAJ}
\begin{cases}
&\epsi z^\epsi_t+H(x,Dz^\epsi,i)+\Theta z^\epsi=\epsi^4 \Delta z^\epsi \quad\rm{in}\  \Tt^d\times (s,1)\times I, \\
&z(x,s,i)=\psi(x,i) \quad \rm{in} \ \Tt^d\times I,
\end{cases}
\end{equation}
where $s\in [0,T]$, $\psi(\cdot,i) \in C^\infty(\Tt^d)$ with $\psi >0$. By comparison, $z^\epsi>0$ in $\Tt^d\times [s,1] \times I$. Next, we multiply \eqref{RAJ} by $\sigma^\epsi$ and \eqref{AJ} by $z^\epsi$, respectively. Adding each other and integrating over $\Tt^d\times I$, we have  
\[\frac{d}{dt}\int_{\Tt^d\times I} z^\epsi \sigma^\epsi \mbox{ }dxdi=0.\]
On the other hand, integrate the above over $[s,1]$ to yield
\begin{equation}\label{qq}
\int_{\Tt^d \times I} \sigma^\epsi(x,s,i)\psi(x,i) \mbox{ }dxdi=z^\epsi(x_0,1,k)>0.
\end{equation}
Since \eqref{qq} holds for any positive $\psi$, $\sigma^\epsi$ is not negative in $\Tt^d\times [0,1] \times I$.

 Next, integrate \eqref{AJ} over $\Tt^d\times I$ and use \eqref{sum}, to get
 \[ \epsi \frac{d}{dt}\int_{\Tt^d\times I} \sigma^\epsi \mbox{ }dxdi=\int_{\Tt^d\times I}-\epsi^4\Delta \sigma^\epsi-\mathrm{div}(\sigma^\epsi D_pH(x,Du^\epsi_2),i))+\Theta \sigma^\epsi \mbox{ }dxdi=0.\]
 Hence, for each $t\in [0,1]$,
 \[\int_{\Tt^d\times I} \sigma^\epsi(x,t,i) \mbox{ }dxdi=\int_{\Tt^d\times I} \sigma^\epsi(x,1,i) \mbox{ }dxdi=1.\]
 \end{proof} 
 
 \subsection{Infimum over holonomic measures}
 In this subsection, we recall the argument about the minimizing problem \eqref{mb}. In the following Proposition, we show that the value of \eqref{mb} is nonnegative. Later, we can see that \eqref{mb} is actually attained and its infimum is zero in Lemma \ref{mather}. 
 
 \begin{proposition}\label{infimum}
 Assume that \ref{convex} holds and the ergodic constant of \eqref{EP} is $0$. Then, we have
 \begin{equation}\label{inf}
 \int_{\Tt^d\times \Rr^d \times I}L(x,q,i)\mbox{ }d\mu \geq 0, 
 \end{equation}
 for all $\mu \in \mathcal{F}$.
 \end{proposition}
 
 \begin{proof}
 Let $v(x,i)$ be a Lipschitz continuous viscosity solution of \eqref{EP} and set $v^\delta$ as defined in \eqref{moli}. 
 Due to \ref{convex} and Jensen's inequality, for all $(x,i)\in \Tt^d\times I$, we get 
 \begin{align*}
 H(x,Dv^\delta(x,i),i)
 &=H\left(x,\int_{\Tt^d}\gamma^\delta(y)Dv(x-y,i)\mbox{ }dy,i\right)\\
 &\leq \int_{\Tt^d}H(x,Dv(x-y,i),i)\gamma^\delta(y)\mbox{ }dy \\
 &\leq \int_{\Tt^d}H(x-y,Dv(x-y,i),i)\gamma^\delta(y)\mbox{ }dy+C\delta \\
 &\leq \int_{\Tt^d}-\Theta v(x-y,i)\gamma^\delta(y)\mbox{ }dy +C\delta=-\Theta v^\delta(x,i)+C\delta.
 \end{align*}
 For any $\mu \in \mathcal{F}$, we have
\begin{align*}
 \int_{\Tt^d\times \Rr^d \times I}C\delta \mbox{ }d\mu 
 &\geq  \int_{\Tt^d\times \Rr^d \times I}H(x,Dv^\delta,i)+\Theta v^\delta \mbox{ }d\mu(x,q,i) \\
 &\geq \int_{\Tt^d\times \Rr^d \times I} -L(x,q,i)+q\cdot Dv^\delta+\Theta v^\delta \mbox{ }d\mu(x,q,i)  \\
  &= \int_{\Tt^d\times \Rr^d \times I}-L(x,q,i) \mbox{ }d\mu(x,q,i),
 \end{align*}
 where we used the property of $\mathcal{F}$ in the last equality. Sending $\delta \to 0$, we get \eqref{inf}.
 \end{proof}

\section{proof of main theorem}
Under the above estimates, we prove the main result. The following argument is introduced in \cite{MT} to study comparison principle with respect to Mather measures for a single equation.
\begin{proof}[Proof of Theorem \ref{result}]
Let $u^{\epsi}_1$ and $u^{\epsi}_2$ be the solution of \eqref{CPR} with $l=1,2$, respectively.
In view of \ref{convex}, we have
\begin{equation}\label{dd}
\epsi (u^{\epsi}_1-u^{\epsi}_2)_t+D_pH(x,Du^{\epsi}_2,i)\cdot D(u^{\epsi}_1-u^{\epsi}_2)+\Theta (u^{\epsi}_1-u^{\epsi}_2)\leq \epsi^4 \Delta( u^{\epsi}_1-u^{\epsi}_2).
\end{equation}
Take $x_0\in \Tt^d$ and $k\in I$. Let $\sigma_k^\epsi(\cdot,i)$ be the solution of \eqref{AJ}. Multiply \eqref{dd} by $\sigma_k^\epsi(\cdot,i)$ and integrate over $\Tt^d$ to obtain
\begin{align*}
0&\geq \int_{\Tt^d} \epsi(u^{\epsi}_1-u^{\epsi}_2)_t(\cdot,i)\sigma_k^\epsi(\cdot,i)+\Theta (u^{\epsi}_1-u^{\epsi}_2)(\cdot,i)\sigma_k^\epsi(\cdot,i) \mbox{ }dx\\
&\:\: -\int_{\Tt^d} \{\mathrm{div}(\sigma_k^\epsi(\cdot,i)D_pH(x,Du^{\epsi}_2(\cdot,i),i))+\epsi^4\Delta \sigma_k^\epsi(\cdot,i)\}(u^{\epsi}_1-u^{\epsi}_2)(\cdot,i) \mbox{ }dx\\
&= \int_{\Tt^d} \epsi(u^{\epsi}_1-u^{\epsi}_2)_t(\cdot,i)\sigma_k^\epsi(\cdot,i)+\Theta (u^{\epsi}_1-u^{\epsi}_2)(\cdot,i)\sigma_k^\epsi(\cdot,i) \mbox{ }dx\\
&\:\: +\int_{\Tt^d} \{ \epsi (\sigma_k^\epsi)_t(\cdot,i)-\Theta \sigma_k^\epsi(\cdot,i)\}(u^{\epsi}_1-u^{\epsi}_2)(\cdot,i) \mbox{ }dx.
\end{align*}
Integrating over $I$ and using \eqref{Theta}, we get
\begin{align*}
0&\geq \int_{\Tt^d\times I} \epsi(u^{\epsi}_1-u^{\epsi}_2)_t(\cdot,i)\sigma_k^\epsi(\cdot,i)+\Theta (u^{\epsi}_1-u^{\epsi}_2)(\cdot,i)\sigma_k^\epsi(\cdot,i) \mbox{ }dxdi\\
&\:\: +\int_{\Tt^d\times I} \{ \epsi (\sigma_k^\epsi)_t(\cdot,i)-\Theta \sigma_k^\epsi(\cdot,i)\}(u^{\epsi}_1-u^{\epsi}_2)(\cdot,i) \mbox{ }dxdi\\
&= \epsi \frac{d}{dt}\int_{\Tt^d\times I} (u^{\epsi}_1-u^{\epsi}_2)(\cdot,i)\sigma_k^\epsi(\cdot,i)\mbox{ }dxdi.
\end{align*}
Hence,
\begin{equation}\label{ff}
\int_{\Tt^d\times I} (u^{\epsi}_1-u^{\epsi}_2)(x,1,i)\sigma_k^\epsi(x,1,i)\mbox{ }dxdi\leq
\int_0^1\int_{\Tt^d\times I} (u^{\epsi}_1-u^{\epsi}_2)(x,t,i)\sigma_k^\epsi(x,t,i)\mbox{ }dxdidt.
\end{equation}

In light of Riesz theorem, there exists $\nu^{\epsi} \in P(\Tt^d \times \Rr^d \times I)$ such that for all $\psi \in C_c(\Tt^d \times \Rr^d \times I)$,
\begin{equation}\label{Riesz}
\int_{\Tt^d \times \Rr^d \times I} \psi(x,p,i)\mbox{ }d\nu^{\epsi}(x,p,i)=\int_0^1 \int_{\Tt^d\times I} \psi(x,Du^\epsi_2(x,t,i),i)\sigma_k^\epsi(x,t,i)\mbox{}dxdidt.
\end{equation}
Then, \eqref{ff} becomes
\begin{equation} \label{gg}
(u^{\epsi}_1-u^{\epsi}_2)(x_0,1,k) \leq \int_{\Tt^d\times \Rr^d \times I} (u^{\epsi}_1-u^{\epsi}_2)\mbox{ }d\nu^{\epsi}.
\end{equation}
Due to Proposition \ref{bernstein}, we can see that $\mathrm{supp}( \nu^\epsi) \subset \Tt^d \times B(0,C)\times I$. There exists $\{\epsi_j\}_{j \in \Nn} \to 0$ such that $ \nu^{\epsi_j} \rightharpoonup \nu \in P(\Tt^d \times \Rr^d \times I)$ as $j \to \infty$ weakly in the sense of measure. We set $\mu \in  P(\Tt^d \times \Rr^d \times I)$ such that the pushfoward measure of $\mu$ associated with $(x,q,i)\mapsto (x,D_qL(x,q,i),i)$ is $\nu$, that is, for all $\phi \in C_c(\Tt^d \times \Rr^d \times I)$,
\begin{equation}\label{mu}
\int_{\Tt^d \times \Rr^d \times I} \phi(x,p,i)\mbox{ }d\nu(x,p,i) =\int_{\Tt^d \times \Rr^d \times I} \phi(x,D_qL(x,q,i),i)\mbox{ }d\mu(x,q,i).
\end{equation}
Later, we prove that $\mu$ is a generalized Mather measure  in Lemma \ref{mather}.

By Proposition \ref{converge}, sending $j \to \infty$ in \eqref{gg}, we get
\begin{equation*}
v_1(x_0,k)-v_2(x_0,k)\leq \int_{\Tt^d\times \Rr^d \times I} (v_1-v_2) \mbox{ }d\mu ,
\end{equation*}
which finishes the proof.
\end{proof}

\begin{lem}\label{mather}
Assume that \ref{convex}-\ref{sym} hold. Let $\mu$ define as \eqref{mu}. Then, $\mu$ is a generalized Mather measure. 
\end{lem}

\begin{proof}
Fix $\phi(\cdot,i) \in C^1(\Tt^d )$. Because $C^\infty(\Tt^d)$ is dense in $C^1(\Tt^d)$, there exists $\{ \phi_n(\cdot,i)\}_{n \in \Nn} \subset C^\infty(\Tt^d)$ satisfying $\phi_n(\cdot,i) \to \phi(\cdot,i)$ in $C^1(\Tt^d)$. Multiply \eqref{AJ} by $\phi_n(\cdot,i)$ and integrate over $\Tt^d\times (0,1) \times I$ to imply
\begin{align}\label{hh}
&\int_0^1 \int_{\Tt^d\times I} D_pH(x,Du^{\epsi}_2(\cdot,i),i)\cdot D\phi_n\sigma_k^\epsi(\cdot,i)\mbox{ }dxdidt
-\int_0^1 \int_{\Tt^d\times I} \phi_n \Theta \sigma_k^\epsi(\cdot,i)\mbox{ }dxdidt\\ \notag
&=-\epsi \int_{\Tt^d\times I} \phi_n\sigma_k^\epsi(x,0,i) \mbox{ }dxdi+\epsi \phi_n(x_0,k)+\epsi^4 \int_0^1 \int_{\Tt^d\times I}\Delta \phi_n \sigma_k^\epsi(\cdot,i)\mbox{ }dxdidt.
\end{align}
By \eqref{Riesz} and \eqref{Theta},
we can rewrite \eqref{hh} as 
\begin{align*}
&\int_{\Tt^d \times \Rr^d \times I} D_pH(x,p,i)\cdot D\phi_n-\Theta \phi_n\mbox{ }d\nu^{\epsi}(x,p,i)\\
&=-\epsi \int_{\Tt^d\times I} \phi_n\sigma_k^\epsi(x,0,i) \mbox{ }dxdi+\epsi \phi_n(x_0,k)+\epsi^4 \int_0^1 \int_{\Tt^d\times I}\Delta \phi_n \sigma_k^\epsi(\cdot,i)\mbox{ }dxdidt.
\end{align*}
Sending $\epsi=\epsi_j \to 0$ and $n \to \infty$, because of Proposition \ref{l1}, the right hand side goes to $0$. Thus, we get
\begin{equation*}
\int_{\Tt^d \times \Rr^d \times I} D_pH(x,p,i)\cdot D\phi-\Theta \phi \mbox{ }d\nu(x,p,i)
=0.
\end{equation*}
On the other hand, by the definition of $\mu$, we have
\[
\int_{\Tt^d \times \Rr^d \times I} D_pH(x,p,i)\cdot D\phi-\Theta \phi \mbox{ }d\nu(x,p,i)=\int_{\Tt^d \times \Rr^d \times I} q\cdot D\phi(x,i)+\Theta \phi(x,i)\mbox{ }d\mu(x,q,i).\]
Hence, $\mu \in \mathcal{F}$ and it suffices to show $\mu$ is a minimizer of \eqref{mb}. We rewrite \eqref{CPR} as  
\begin{align*}
&\epsi (u^{\epsi}_2)_t(\cdot,i)+D_pH(x,Du^{\epsi}_2(\cdot,i),i)\cdot Du^{\epsi}_2(\cdot,i)-\epsi^4 \Delta u^{\epsi}_2(\cdot,i)+\Theta u^{\epsi}_2(x,i)\\
&=D_pH(x,Du^{\epsi}_2(\cdot,i),i)\cdot Du^{\epsi}_2(\cdot,i)-H(x,Du^{\epsi}_2(\cdot,i),i).
\end{align*}
Multiply this by $\sigma_k^\epsi(\cdot,i)$ and integrate over $\Tt^d\times(0,1)\times I$ to yield
\begin{align*}
&\epsi u^{\epsi}_2(x_0,1,k)-\epsi \int_{\Tt^d\times I} u^{\epsi}_2(x,0,i)\sigma_k^\epsi(x,0,i)\mbox{ }dxdi\\
&=\int_0^1 \int_{\Tt^d\times I}\{D_pH(x,Du^{\epsi}_2(\cdot,i),i)\cdot Du^{\epsi}_2(\cdot,i)-H(x,Du^{\epsi}_2(\cdot,i),i)\}\sigma_k^\epsi(\cdot,i) \mbox{ }dxdidt.
\end{align*}
In light of \eqref{Riesz} and letting $\epsi=\epsi_j \to 0$, we get
\begin{equation}\label{zerovalue}
0=\int_{\Tt^d\times \Rr^d \times I}(D_pH(x,p,i)\cdot p-H(x,p,i))\mbox{ }d\nu(x,p,i)=\int_{\Tt^d\times \Rr^d \times I}L(x,q,i)\mbox{ }d\mu(x,q,i).
\end{equation}
In view of \eqref{inf}, $\mu$ is a minimizer of \eqref{mb}.
\end{proof}

Combining Proposition \ref{infimum} and \eqref{zerovalue}, we can see the following fact.
\begin{corollary}\label{mbp}
Assume that \ref{convex}-\ref{sym} holds and the ergodic constant of \eqref{EP} is $0$. Then, we have
 \begin{equation*}
 \int_{\Tt^d\times \Rr^d \times I}L(x,q,i)\mbox{ }d\mu = 0, 
 \end{equation*}
 for all $\mu \in \mathcal{\tilde M}$.
\end{corollary}

\section{Example of Mather measure}\label{example}
The following weakly coupled PDE is an example introduced in \cite{davi}, \cite{cami2} and \cite{MT2}, for instance. 
Here, we use this to show an example of a generalized Mather measure.
Let $f:\Tt^d \times I \to [0,\infty)$ satisfying $\bigcap_{i=1}^m A_i \neq \emptyset$, where $A_i=\{x\in \Tt^d\: |\: f(x,i)=0\}$ for $i\in I$. Then, consider
\begin{align*}
\frac{1}{2}|Dv(x,i)|^2+\sum_{j=1}^m c_{ij}(v(x,i)-v(x,j))=f(x,i) \quad \mathrm{in}\ \Tt^d\times I,
\end{align*}where $v: \Tt^d \times I \to \Rr$ is unknown. Note that  in this setting the ergodic constant is zero. 

The above problem corresponds to \eqref{EP}, when $H(x,p,i)=\frac{1}{2}|p|^2-f(x,i)$.

\begin{proposition}
Let $x_0\in \bigcap_{i=1}^m A_i$. Then, $\mu:=  \frac{1}{m} \sum_{i=1}^m\delta_{(x_0,0,i)}$ is a generalized Mather measure.
\end{proposition}

\begin{proof}
At first, we show that $\mu \in \mathcal{F}$. For $\phi(\cdot,i) \in C^1(\Tt^d)$, we get
\begin{align*}
\int_{\Tt^d \times \Rr^d \times I}q\cdot D\phi(x,i)d\mu(x,q,i)
=\frac{1}{m}\sum_{i=1}^m 0\cdot \phi(x_0,i)=0,
\end{align*}
and using \eqref{sum}, we can see
\begin{align*}
\int_{\Tt^d \times \Rr^d \times I}\Theta \phi(x,i)\mbox{ }d\mu(x,q,i)
=\frac{1}{m}\sum_{i=1}^m\Theta \phi(x_0,i)=0,
\end{align*}
which implies $\mu \in \mathcal{F}$. On the other hand, we have
\begin{align*}
\int_{\Tt^d \times \Rr^d \times I} L(x,q,i) \mbox{ }d\mu(x,q,i)&=\int_{\Tt^d \times \Rr^d \times I} \frac{1}{2}|q|^2+f(x,i) \mbox{ }d\mu(x,q,i)\\
&=\frac{1}{m} \sum_{i=1}^m f(x_0,i)=0.
\end{align*}
 In view of Proposition \ref{infimum}, $\mu$ is a minimizer of \eqref{mb}.
\end{proof}

\begin{remark}
We recall the case of single equation, that is $m=1$. For $n\in \Nn$, let $a_k\geq 0$ be constants satisfying $\sum_{k=1}^n a_k=1$. Then, $\mu=\sum_{k=1}^n a_k \delta_{(x_k,0,1)}$ is a Mather measure, where $x_k\in A_1$ for $k\in \{1,2,...,n\}$.
However, in the system case, because we generalize the holonomic condition for the minimizing problem \eqref{mb}, the above convex combinations are not generalized Mather measures in general.
\end{remark}

\vspace{3mm}
{\bf Acknowledgement.} The author would like to thank Professor Hiroyoshi Mitake for his helpful comments and suggestions.

\bibliographystyle{plain}

\bibliography{vakonomic}

\end{document}